\documentclass{article}

\usepackage{amstext,amsfonts,amsmath,amssymb,amsthm}
\usepackage{amscd}
\usepackage{fancyhdr}    

\newcommand \Rn{\mathbb R^n}

\renewcommand \>{\rangle}

\newtheorem*{thm*}{Theorem}

\begin{document}

\title{A note on Raki\'{c} Duality principle for Osserman manifolds}



\title{A note on the duality principle and  Osserman condition \hspace{.25mm}
\thanks{\,Z. R. is partially supported by the
\textsl{Serbian Ministry of Education and Science},
\textsf{project No. 174012.} Y.N. is partially supported by the
\textsl{La Trobe University DGS grant}.}}

\author{\textbf{Y. Nikolayevsky}\hspace{.25mm}
 \thanks{\,E-mail address: y.nikolayevsky@latrobe.edu.au}
  \\[3pt] \normalsize{Department of Mathematics and Statistics, La Trobe University,}\\[2pt] \normalsize{Victoria, 3086, Australia} \vspace{5mm}
  \\ \textbf{Z. Raki\'{c}}\hspace{.25mm}
      \thanks{\,E-mail address: zrakic@matf.bg.ac.rs
      \newline {\sl{Mathematical Subject Classification 2000:}} 53B30, 53C50.
}
  \\[3pt] \normalsize{Faculty of Mathematics, University of Belgrade, Serbia}}

\date{ }
\maketitle



\begin{abstract}

\noindent In this note we prove that for a Riemannian manifold the Osserman pointwise condition is equivalent to the Raki\'{c} duality principle.

\end{abstract}

\section{Introduction}
\label{s:intro}
\noindent Let $\mathcal{R}$ be an algebraic curvature tensor on a Euclidean space $\Rn$ and let for $X \in \Rn$, $\mathcal{R}_X:Y\mapsto \mathcal{R}(Y,X)X$ be the corresponding Jacobi operator. An algebraic curvature tensor $\mathcal{R}$ is called {\it Osserman}, if the spectrum of the Jacobi operator $\mathcal{R}_X$ does not depend on the choice of a unit vector $X \in \Rn$. \vspace{2mm}

\noindent Let $M^n$ be a Riemannian manifold, $R$ its curvature tensor and $R_X$ the corresponding Jacobi operator. It is well known that the properties of $R_X$ are intimately related with the underlying geometry of the manifold. The manifold $M^n$ is called {\it pointwise Osserman} if $R$ is Osserman at every point $p \in M^n$, and is called {\it globally Osserman} if the spectrum of $R_X$ is the same for all $X$ in the unit tangent bundle of $M^n$. Locally two-point homogeneous spaces are globally Osserman, since the isometry group of each of these spaces is transitive on its unit tangent bundle. Osserman \cite{O} conjectured that the converse is also true. This gives a very nice characterisation of local two-point homogeneous spaces in terms of the geometry of the Jacobi operator.\vspace{2mm}

\noindent At present, the Osserman Cojecture is almost completely solved by the results of Chi \cite{C}, who proved the Conjecture in dimensions $n\neq 4k$, $k>1$ and $n=4$, and the first auhtor \cite{N1, N2, N3}, who proved it in all the remaining cases, except for some cases in dimension $n=16$. \vspace{2mm}

\noindent One of the crucial steps in the existing proofs of the Osserman Conjecture is the following \textsf{Raki\' c duality principle} \cite{R}: \begin{quote}
\textsl{Suppose $\mathcal{R}$ is an Osserman algebraic curvature tensor and $X, Y$ are unit vectors. Then $Y$ is an unit eigenvector of $\mathcal{R}_X$ if and only if $X$ is an unit eigenvector of $\mathcal{R}_Y$ (with the same eigenvalue)}.
\end{quote}
The duality principle is extended to the pseudo-Riemannian settings in \cite{AR}.\vspace{2mm}

\section{Equivalence of duality principle and Osserman pointwise condition}
\noindent  Recently, for  an algebraic curvature tensor in Riemannian signature, M.~Brozos-V\'{a}zquez and E.~Merino \cite{BM} proved the equivalence of the Osserman condition and the duality principle for spaces of dimension less than $5.$ We show that this holds in an arbitrary dimension.\vspace{3mm}

\begin{thm*}
The following two conditions for an algebraic curvature tensor $\mathcal{R}$ in Riemannian signature are equivalent:
\begin{itemize}
  \item[{\rm (a)}]
  $\mathcal{R}$ satisfies the  duality principle;
  \item[{\rm (b)}]
  $\mathcal{R}$ is Osserman.
\end{itemize}
\end{thm*}

\begin{proof}
The implication (b) $\implies$ (a) is proved in \cite{R}.\vspace{1.2mm}

\noindent To establish the converse, consider the characteristic polynomial $\chi_X(t)$ of the Jacobi operator $\mathcal{R}_X$, where $X$ is a unit vector. As the coefficients of $\chi_X$ are analytic function on the unit sphere $S \subset \Rn$, there is an open and dence subset $S' \subset S$ such that for all $X \in S'$ the number and the multiplicity of the eigenvalues of $\mathcal{R}_X$ are constant, the eigenvalues are analytic functions of $X$, and the eigendistributions of $J_X$ are analytic (viewed as the curves in the appropriate Grassmannians), see \cite{Re}, \cite{K}. \vspace{1mm}

\noindent Let $X \in S'$ and let $Y \in S$ be orthogonal to $X$. Suppose $\lambda_0$ is an eigenvalue of $\mathcal{R}_X$ with a unit eigenvector $e_0$. For small $\phi$, the vector $\cos \phi X + \sin \phi Y$ belongs to $S'$, so there exist a differentiable (in fact, analytic) eigenvalue function $\lambda(\phi)$ of the operator $\mathcal{R}_{\cos \phi X + \sin \phi Y}$ such that $\lambda(0)=\lambda_0$ and a differentiable unit vector function $e(\phi)$, a section of the $\lambda(\phi)$-eigenspace of $\mathcal{R}_{\cos \phi X + \sin \phi Y}$ such that $e(0)=e_0$. Differentiating the equation
\begin{equation*}
\mathcal{R}(\cos \phi X + \sin \phi Y,e(\phi),\cos \phi X + \sin \phi Y,e(\phi))=\lambda(\phi)
\end{equation*}
at $\phi=0$ we obtain
\begin{equation*}
2\mathcal{R}(Y,e_0,X, e_0)+2\mathcal{R}(X,e_0,X,e'(0))=\lambda'(0).
\end{equation*}
But $\mathcal{R}(X,e_0,X,e'(0))=\lambda_0\<e_0,e'(0)\>=0$ and also $\mathcal{R}(Y,e_0,X, e_0)=\lambda_0\<X,Y\>=0$, by duality. It follows that the eigenvalues of $\mathcal{R}_X$ are constant on every connected component of $S'$. Then the coefficients of $\chi_X(t)$ are constant on the whole unit sphere $S$, which implies that $\mathcal{R}$ is Osserman.
\end{proof}

\end{document}